\documentclass{amsart}
\usepackage{amssymb}
\usepackage{amscd}
\usepackage{amsfonts}
\usepackage{latexsym}
\usepackage[all]{xy}
\usepackage{verbatim}

\newtheorem{theorem}{Theorem}[section]
\newtheorem{lemma}[theorem]{Lemma}
\newtheorem{proposition}[theorem]{Proposition}
\newtheorem{corollary}[theorem]{Corollary}
\theoremstyle{definition}
\newtheorem{definition}[theorem]{Definition}
\newtheorem{example}[theorem]{Example}

\newtheorem{remark}[theorem]{Remark}
\newtheorem{note}[theorem]{Note}
\newcommand{\id}{\text{id}}

\newcommand{\FPdim}{\text{FPdim}}

\newcommand{\Fun}{\text{Fun}}

\renewcommand{\Vec}{\text{Vec}}

\newcommand{\Hom}{\text{Hom}}

\newcommand{\Rep}{\text{Rep}}

\newcommand{\rev}{\text{rev}}

\newcommand{\BN}{\mathbb{N}}

\newcommand{\B}{\mathcal{B}}
\newcommand{\C}{\mathcal{C}}
\newcommand{\K}{\mathcal{K}}
\newcommand{\D}{\mathcal{D}}
\newcommand{\E}{\mathcal{E}}

\newcommand{\Z}{\mathcal{Z}}

\renewcommand{\L}{\mathcal{L}}
\newcommand{\M}{\mathcal{M}}
\newcommand{\A}{\mathcal{A}}

\newcommand{\V}{\mathcal{V}}
\renewcommand{\O}{\mathcal{O}}

\newcommand{\be}{\mathbf{1}}

\newcommand{\equivalent}{\Leftrightarrow}

\renewcommand{\be}{\mathbf{1}}

\newcommand{\BZ}{{\mathbb Z}}

\newcommand{\bt}{\boxtimes}
\newcommand{\ot}{\otimes}

\begin{document}
\title{Group-theoretical properties of nilpotent modular categories}

\author{Vladimir Drinfeld}
\address{V.D.: Department of Mathematics,
University of Chicago, Chicago, IL 60637, USA}
\email{drinfeld@math.uchicago.edu}

\author{Shlomo Gelaki}
\address{S.G.: Department of Mathematics, Technion-Israel Institute of
Technology, Haifa 32000, Israel}
\email{gelaki@math.technion.ac.il}

\author{Dmitri Nikshych}
\address{D.N.: Department of Mathematics and Statistics,
University of New Hampshire,  Durham, NH 03824, USA}
\email{nikshych@math.unh.edu}

\author{Victor Ostrik}
\address{V.O.: Department of Mathematics,
University of Oregon, Eugene, OR 97403, USA}
\email{vostrik@math.uoregon.edu}

\dedicatory{To Yuri Ivanovich Manin on his 70th birthday}
\date{March 31, 2007}

\begin{abstract}
We characterize a natural class of modular categories
of prime power Frobenius-Perron dimension as representation
categories of twisted doubles of finite $p$-groups.
We also show that a nilpotent braided fusion category
$\C$ admits an analogue of the Sylow decomposition. If
the simple objects of $\C$ have
integral Frobenius-Perron dimensions  then $\C$ is group-theoretical
in the sense of \cite{ENO}. As a consequence,
we obtain that semisimple quasi-Hopf algebras of prime power
dimension are group-theoretical. Our arguments are based
on a reconstruction of twisted group doubles from Lagrangian
subcategories of modular categories (this is reminiscent to the
characterization of doubles of quasi-Lie bialgebras in terms
of Manin pairs given in \cite{Dr}).
\end{abstract}
\maketitle

\section{introduction}

In this paper we work over an algebraically closed field $k$ of characteristic $0$.

By a {\em fusion category} we mean a $k$-linear semisimple rigid
tensor category $\C$ with finitely many isomorphism classes of
simple objects, finite dimensional spaces of morphisms, and such
that the unit object $\be$ of $\C$ is simple.  We refer the reader to
\cite{ENO} for a general theory of such categories. A fusion category is
{\em pointed} if all its simple objects are invertible. A pointed fusion category
is equivalent to $\Vec_G^\omega$, i.e., the category of $G$-graded vector
spaces  with the associativity constraint given by some cocycle
$\omega\in Z^3(G,\, k^\times)$ (here $G$ is a finite group).

\subsection{Main results}
\label{main results}
\begin{theorem}
Any braided nilpotent fusion category has a unique decomposition into
a tensor product of braided fusion categories whose Frobenius-Perron dimensions
are powers of distinct primes.
\end{theorem}

The notion of {\em nilpotent\,} fusion category was introduced in  \cite{GN};
we recall it in Subsection \ref{nilpprelim}.
Let us mention that the representation
category $\Rep(G)$ of a finite group $G$ is
nilpotent if and only if $G$ is nilpotent.  It is also known that
fusion categories of prime power Frobenius-Perron dimension are nilpotent
\cite{ENO}. On the other hand, $\Vec_G^\omega$ is nilpotent for {\em any\,}
$G$ and $\omega$. Therefore it is {\em not\,} true that any nilpotent fusion category is
a tensor product of fusion categories of prime power dimensions.

\begin{theorem}
\label{Thm 1.2}
A modular category $\C$ with integral dimensions of simple objects
is nilpotent if and only if there exists a pointed modular category $\M$ such that
$\C\boxtimes\M$ is equivalent,
as a braided tensor category, to the center of a fusion category of the form
$\Vec_G^\omega$ for a finite nilpotent group $G$.
\end{theorem}

We emphasize here that in general the equivalence in Theorem~\ref{Thm 1.2}  does not respect
the spherical structures (equivalently, twists) of the categories involved and
thus {\em is not} an equivalence of modular categories. Fortunately, this is not a
very serious complication since the spherical structures on $\C$ are easy to classify:
it is well known that they are in bijection with the objects $X\in \C$ such that
$X\otimes X=\be$, see \cite{RT}.

The category $\M$ in Theorem~\ref{Thm 1.2} is not uniquely determined by $\C$. However,
there are canonical ways to choose $\M$. In particular, one  can  always
make a canonical ``minimal" choice for $\M$ such that $\dim(\M)=\prod_pp^{\alpha_p}$ with
$\alpha_p\in \{ 0,1,2\}$ for odd $p$ and $\alpha_2\in\{0,1,2,3\}$, see Remark~\ref{exponent leq 2}.

\begin{theorem} \label{pintro}
A modular category $\C$
is braided equivalent to the center of a fusion category of the form
$\Vec_G^\omega$ with $G$ being a finite $p$-group if and only if it has the
following
properties:
\begin{enumerate}
\item[(i)] the Frobenius-Perron dimension of $\C$ is $p^{2n}$ for some $n\in \mathbb{Z^+}$,
\item[(ii)] the dimension of every simple object of $\C$ is an integer,
\item[(iii)] the multiplicative central charge of $\C$ is $1$.
\end{enumerate}
\end{theorem}

See Subsection \ref{gauss} for the definition of multiplicative central charge.
In order to avoid confusion we note that our definition of multiplicative
central charge is different from the definition of central charge of a modular functor
from \cite[5.7.10]{BK}; in fact, the central charge from \cite{BK} equals to the square of our central charge.



\begin{remark} \label{rintro}
If $p\ne 2$ then it is easy to see that (i) implies (ii) (see, e.g., \cite{GN}).
\end{remark}

\subsection{Interpretation in terms of group-theoretical
fusion categories and semisimple quasi-Hopf algebras.}
The notion of group-theoretical fusion category was introduced
in \cite{ENO, O1}.
Group-theoretical categories form a
large class of well-understood fusion categories which can be
explicitly constructed from finite group data (which justifies the
name). For example, as far as we know, all currently known
semisimple Hopf algebras have group-theoretical representation
categories (however, there are semisimple quasi-Hopf algebras
whose representation categories  are not group-theoretical, see
\cite{ENO}).

\begin{theorem}
Let $\C$ be a fusion category
such that  all objects of $\C$ have integer dimension
and such that its center $\Z(\C)$ is nilpotent.
Then $\C$ is group-theoretical.
\end{theorem}

\begin{remark}
A consequence of this theorem is the following statement: every semisimple
(quasi-)Hopf algebra of prime power dimension is
group-theoretical in the sense of \cite[Definition 8.40]{ENO}.
This provides a partial answer to  a question asked in \cite{ENO}.
\end{remark}

\subsection{Idea of the proof}


We describe here the main steps in the proof of Theorem \ref{pintro}.
First we characterize centers
of pointed fusion categories in terms of Lagrangian subcategories and show
that a modular category $\C$ is equivalent to the representation
category of a twisted group double if and only if it has a
Lagrangian (i.e., maximal isotropic) subcategory of dimension
$\sqrt{\dim(\C)}$. This result is reminiscent to the characterization
of doubles of quasi-Lie bialgebras in terms of Manin pairs
\cite[Section 2]{Dr}.

Thus we need to show that a category satisfying the assumptions of Theorem \ref{pintro}
contains a Lagrangian subcategory. The proof is inspired by the following
result for nilpotent metric Lie algebras (i.e, Lie algebras with
an invariant non-degenerate scalar product) which can be derived
from \cite{KaO}: if $\mathfrak{g}$ is a nilpotent metric Lie
algebra of even dimension then $\mathfrak{g}$ contains an abelian
ideal $\mathfrak{k}$, which is Lagrangian (i.e., such that
$\mathfrak{k}^\perp =\mathfrak{k}$). The relevance of metric Lie
algebras to our considerations is explained by the fact that they
appear in \cite{Dr} as classical limits of quasi-Hopf algebras.
In fact, our proof is a ``categorification" of the proof of the above result. Thus we
need some categorical versions of linear algebra constructions involved in this proof.
Remarkably, the categorical counterparts exist for all notions required. For example
the notion of orthogonal complement in a metric Lie algebra is replaced by the
notion of centralizer in a modular tensor category introduced by M.~M\"uger \cite{Mu2}.



\subsection{Organization of the paper} Section 2 is devoted to
preliminaries on fusion categories, which include nilpotent fusion
categories, (pre)modular categories, centralizers, Gauss sums and
central charge, and Deligne's classification
of symmetric fusion categories.

In Section 3 we define the notions of isotropic and Lagrangian
subcategories of a premodular category $\C$, generalizing the
corresponding notions for a metric group (which is, by definition,
a finite abelian group with a quadratic form). We then recall a
construction, due  to A.~ Brugui\`{e}res \cite{Br} and M.~M\"uger
\cite{Mu1}, which associates to a premodular category $\C$ the
``quotient'' by its centralizer, called a {\em modularization}. We
prove in Theorem \ref{tau after mod} an invariance property of the
central charge with respect to the modularization. This result
will be crucial in the proof of Theorem \ref{main2}. We also study
properties of subcategories of modular categories and explain  in
Proposition~\ref{canonical modularization} how one can use maximal
isotropic subcategories of a modular category $\C$ to canonically
measure a failure of $\C$ to be hyperbolic (i.e., to contain a
Lagrangian subcategory).

In Section 4 we characterize hyperbolic modular  categories.
More precisely, we show in Theorem \ref{main1} that for
a modular category  $\C$ there is a bijection between
Lagrangian subcategories of $\C$ and
braided tensor equivalences $\C \xrightarrow{\sim} \mathcal{Z}(\Vec_G^\omega)$
(where $G$ is a finite group, $\omega\in Z^3(G, K^\times)$, and
$\mathcal{Z}(\Vec_G^\omega)$ is the center of $\Vec_G^\omega$).
Note that the category  $\mathcal{Z}(\Vec_G^\omega)$ is equivalent to
$\Rep(D^\omega(G))$ - the representation category of the twisted
double of $G$ \cite{DPR}.

We then prove in Theorem~\ref{dichotomy} 
that if $\C$ is a modular category such that
$\dim(\C)= n^2,\, n\in \mathbb{Z}^+$, 
the central charge of $\C$ equals $1$,
and $\C$ contains a symmetric subcategory of dimension $n$, then
either $\C$ is equivalent to the representation category of a
twisted double of a finite group or $\C$ contains an object with
non-integer dimension.

We also give a criterion for a modular category $\C$ to be 
group-theoretical. Namely, we show in 
Corollary~\ref{group theor criterion} that $\C$ is group-theoretical
if and only if there is an isotropic subcategory $\E \subset \C$ such
that $(\E')_{ad} \subseteq \E$.

In Section 5 we study pointed modular $p$-categories. We give a
complete list of such categories which do not contain non-trivial
isotropic subcategories and analyze the values of their central
charges. We then prove in Proposition \ref{basis} that a
nondegenerate metric $p$-group $(G,\, q)$ with central charge $1$
such that $|G|=p^{2n},\, n\in\mathbb{Z}^+$, contains a Lagrangian
subgroup.

Section 6 is devoted to nilpotent modular categories. There we
give proofs of our main results stated in \ref{main results}
above. They are contained in Theorem~\ref{main2},
Theorem~\ref{main3}, Corollary~\ref{cor0}, Theorem~\ref{cor2}, and
Theorem~\ref{cor3}.


\subsection{Acknowledgments} The research of V.~Drinfeld was supported by NSF grant
DMS-0401164.
The research of D.~Nikshych was supported by the NSF grant DMS-0200202 and
the NSA grant H98230-07-1-0081.   The research of V.~Ostrik was supported by NSF grant
DMS-0602263.  S.~Gelaki is grateful to the departments of mathematics
at the University of New Hampshire and MIT for their warm hospitality
during his Sabbatical. The authors are grateful to Pavel Etingof for useful discussions.







\section{Preliminaries}

Throughout the paper we work over an algebraically closed field
$k$ of characteristic $0$. All categories considered in this paper
are finite, abelian, semisimple, and $k$-linear.

\subsection{Fusion categories}


For a fusion category $\C$ let $\O(\C)$ denote the set of isomorphism classes
of simple objects.


Let $\C$ be a fusion category. Its Grothendieck ring $K_0(\C)$ is
the free $\mathbb{Z}$-module generated by the isomorphism classes
of simple objects of $\C$ with the multiplication coming from the
tensor product in $\C$. The Frobenius-Perron dimensions of objects
in $\C$ (respectively, $\FPdim(\C)$) are defined as the
Frobenius-Perron dimensions of their images in the based ring
$K_0(\C)$ (respectively, as $\FPdim(K_0(\C))$),
see \cite[8.1]{ENO}. For a semisimple
quasi-Hopf algebra $H$ one has $\FPdim(X) = \dim_k(X)$ for all $X$
in $\Rep(H)$, and so $\FPdim(\Rep(H)) = \dim_k(H)$.

A fusion category is {\em pointed} if all its simple objects are
invertible.

By a {\em fusion subcategory} of a fusion category $\C$ we
understand a full tensor subcategory of $\C$. An example of a
fusion subcategory is the maximal pointed subcategory $\C_{pt}$
generated by the invertible objects of $\C$.

A fusion category $\C$ is {\em pseudo-unitary} if its categorical
dimension $\dim(\C)$ coincides with its Frobenius-Perron
dimension, see \cite{ENO} for details. In this case $\C$ admits a
canonical spherical structure (a tensor  isomorphism between the
identity functor of $\C$ and the second duality functor) with
respect to which categorical dimensions of objects coincide with
their Frobenius-Perron dimensions \cite[Proposition 8.23]{ENO}.
The fact important for us in this paper
is that a fusion category of an integer Frobenius-Perron dimension
is automatically pseudo-unitary \cite[Proposition 8.24]{ENO}.

Let $\C$ and $\D$ be fusion categories. Recall that for a tensor
functor $F: \C\to \D$ its image $F(\C)$ is the fusion subcategory
of $\D$ generated by all simple objects $Y$ in $\D$ such that $Y
\subseteq F(X)$ for some simple $X$ in $\C$. The functor $F$ is
called {\em surjective} if $F(\C) =\D$.

\subsection{Nilpotent fusion categories}\label{nilpprelim}

For a fusion category $\C$ let $\C_{ad}$ be the trivial component
in the universal grading of $\C$ (see \cite{GN}). Equivalently,
$\C_{ad}$ is the smallest fusion subcategory of $\C$ which
contains all the objects $X\otimes X^*$, $X\in \O(\C)$.

For a fusion category $\C$ we define $\C^{(0)}= \C,\, \C^{(1)} =
\C_{ad},$ and $\C^{(n)} = (\C^{(n-1)})_{ad}$ for every integer
$n\geq 1$. The non-increasing sequence of fusion subcategories of
$\C$
\begin{equation}
\C= \C^{(0)} \supseteq  \C^{(1)} \supseteq \cdots
\supseteq \C^{(n)} \supseteq \cdots
\end{equation}
is called the {\em upper central series} of $\C$.
We say that a fusion  category $\C$ is {\em nilpotent}
if every non-trivial subcategory of $\C$ has a non-trivial group
grading, see \cite{GN}. Equivalently,
$\C$ is nilpotent if its
upper central series converges to $\Vec$ (the category of finite
dimensional $k-$vector spaces), i.e., $\C^{(n)} =\Vec$ for some
$n$. The smallest such $n$ is called the {\em nilpotency class} of
$\C$. If $\C$ is nilpotent then every fusion subcategory
$\mathcal{E}\subset \C$ is nilpotent, and if $F: \C\to \D$ is a
surjective tensor functor, then $\D$ is nilpotent (see \cite{GN}).

\begin{example}
\begin{enumerate}
\item Let $G$ be a finite group and $\C=\Rep(G)$. Then $\C$ is
nilpotent if and only if $G$ is nilpotent.
\item Pointed categories are precisely
the nilpotent fusion categories of nilpotency class~$1$. A typical
example of a pointed category is $\Vec_G^{\omega}$, the category
of finite dimensional vector spaces graded by a finite group $G$
with the associativity constraint determined by $\omega\in
Z^3(G,k^{\times})$.
\end{enumerate}
\end{example}

In this paper we are especially interested in the following class
of nilpotent fusion categories.

\begin{example}
Let $p$ be a prime number. Any category of dimension $p^n, n\in
\mathbb{Z}$, is nilpotent  by \cite[Theorem 8.28]{ENO}.  For
representation categories of semisimple Hopf algebras of dimension
$p^n$ this follows from a result of A.~Masuoka~\cite{Ma1}.
\end{example}

By \cite{GN}, a nilpotent fusion category comes from  a sequence
of gradings, in particular it has an integer Frobenius-Perron
dimension. It follows from results of \cite{ENO} that a nilpotent
fusion category $\C$ is pseudounitary.

\subsection{Premodular categories and modular categories}

Recall that a {\em braided} tensor  category $\C$ is a  tensor
category equipped with a natural isomorphism $c: \otimes \cong
\otimes^\rev$ satisfying the hexagon diagrams \cite{JS}.
Let $c_{XY} : X \ot Y \cong  Y\otimes X$ with $ X,Y\in \C$ denote the components of $c$.

A {\em balancing transformation}, or a {\em twist}, on a braided category
$\C$ is a natural automorphism $\theta: \id_\C\to \id_\C$
satisfying  $\theta_{\be} = \id_{\be}$ and
\begin{equation}
\label{balancing ax}
\theta_{X\ot Y} = (\theta_X\ot\theta_Y) c_{YX} c_{XY}.
\end{equation}
A braided fusion category $\C$ is called {\em premodular}, or {\em ribbon}, if
it has a twist $\theta$ satisfying $\theta_X^* = \theta_{X^*}$ for all
objects $X\in\C$.

The {\em $S$-matrix} of a premodular category $\C$  is
$S=\{s_{XY}\}_{X,Y\in \O(\C)}$, where $s_{XY}$ is the quantum
trace of $c_{YX}c_{XY}$, see \cite{T}. Equivalently, the
$S$-matrix can be defined as follows. For all $X,Y,Z \in \O(\C)$
let $N_{XY}^Z$ be the multiplicity of $Z$ in $X\ot Y$.  For every
object $X$ let $d(X)$ denote its quantum dimension. Then
\begin{equation}
\label{sij}
s_{XY} = \theta_X^{-1}\theta_Y^{-1} \sum_{Z\in \O(\C)}
\, N_{XY}^Z \theta_Z d(Z).
\end{equation}

The {\em categorical dimension} of $\C$ is defined by
\begin{equation}
\dim(\C) =\sum_{X\in \O(\C)}\, d(X)^2.
\end{equation}
One has $\dim(\C)\neq 0$ \cite[Theorem 2.3]{ENO}.

\begin{note}
\label{note: sph when k=C} Below we consider only fusion
categories with integer Frobenius-Perron dimensions of objects.
Any such category $\C$ is pseudo-unitary (see 2.1). In
particular, if $\C$ is braided then it has a canonical twist,
which we will always assume chosen.
\end{note}

A premodular category $\C$ is called {\em modular} if the
$S-$matrix is invertible.

\begin{example} \label{center} For any fusion category $\C$ its {\em center}
$\mathcal{Z}(\C)$ is defined as the category whose objects are
pairs $(X, c_{X, -})$, where $X$ is an object of $\C$ and
$c_{X,-}$ is a natural family of isomorphisms $c_{X, V} : X\ot V
\cong V \ot X$ for all objects $V$ in $\C$ satisfying certain
compatibility conditions (see e.g., \cite{Kass}). It is known that
the center of a pseudounitary category is modular.
\end{example}

\subsection{Pointed modular categories and metric groups}\label{CGQ}



Let $G$ be a finite  abelian group. Pointed premodular categories
$\C$ with the group of simple objects isomorphic to $G$ (up to a
braided equivalence) are in the natural bijection with {\em
quadratic forms} on $G$ with values in the multiplicative group
$k^*$ of the base field. Here a quadratic form $q: G\to k^*$ is a
map such that $q(g^{-1})=q(g)$ and
$b(g,h):=\frac{q(gh)}{q(g)q(h)}$ is a symmetric bilinear form,
i.e., $b(g_1g_2,h)= b(g_1,h)b(g_2,h)$ for all $g_1, g_2, h\in G$.
Namely, for $g\in G$ the value of $q(g)$ is the braiding
automorphism of $g\otimes g$ (here by abuse of notation $g$
denotes the object of $\C$ corresponding to $g\in G$). See
\cite[Proposition 2.5.1]{Q} for a proof that if two categories
$\C_1$ and $\C_2$ produce the same quadratic form then they are
braided equivalent (Quinn proves less canonical but equivalent
statement). We will denote the category corresponding to a group
$G$ with quadratic form $q$ by $\C(G,q)$ and call the pair $(G,q)$
a {\em metric group}. The category $\C(G,q)$ is pseudounitary and
hence has a spherical structure such that dimensions of all simple
objects equal to 1; hence the categories $\C(G,q)$ always have  a
canonical ribbon structure. The category $\C(G,q)$ is modular if
and only if the bilinear form $b(g,h)$ associated with $q$ is
non-degenerate (in this case we will say that the corresponding
metric group is {\em non-degenerate}).

\subsection{Centralizers}
Let $\mathcal{K}$ be a fusion subcategory of a
braided fusion category $\C$. In \cite{Mu1, Mu2} M.~M\"uger
introduced the {\em centralizer} $\mathcal{K}'$ of $\mathcal{K}$,
which is the fusion subcategory of $\C$ consisting of all the
objects $Y$ satisfying
\begin{equation}
\label{monodromy} c_{YX}c_{XY} =\id_{X\ot Y} \qquad \mbox{ for all
objects } X \in \K.
\end{equation}
If \eqref{monodromy} holds  we will say that objects $X$ and $Y$
{\em centralize} each other. In the case of a ribbon category
$\C$, condition \eqref{monodromy} is equivalent to $s_{XY} = d(X)
d(Y)$, see  \cite[Proposition 2.5]{Mu2}. Note that in the case of
a pointed modular category the centralizer corresponds to the
orthogonal complement. The subcategory $\C'$ of $\C$ is called the
{\em transparent} subcategory of $\C$ in \cite{Br, Mu1}.

For any fusion subcategory $\mathcal{K}\subseteq \C$ of a braided
fusion category $\C$ let $\mathcal{K}^{co}$ be the {\em
commutator} of $\mathcal{K}$ \cite{GN}, i.e., the fusion
subcategory of $\C$ spanned by  all simple objects $X\in\C$ such
that $X\ot X^*\in \mathcal{K}$. For example, if $\C=\Rep(G)$, $G$
a finite group, then any fusion subcategory $\mathcal{K}$ of $\C$
is of the form $\mathcal{K}=\Rep(G/N)$ for some normal subgroup
$N$ of $G$, and $\mathcal{K}^{co} = \Rep(G/[G, N])$ (see
\cite{GN}). It follows from the definitions that
$(\mathcal{K}^{co})_{ad} \subseteq \mathcal{K} \subseteq
(\mathcal{K}_{ad})^{co}$.

Let $\mathcal{K}$ be a fusion subcategory of a pseudounitary modular
category $\C$. It was shown in \cite{GN} that
\begin{equation}
(\mathcal{K}_{ad})'=(\mathcal{K}')^{co}.
\end{equation}


It was shown in \cite[Theorem 3.2]{Mu2} that for a fusion
subcategory  $\K$ of a modular category $\C$ one has
$\mathcal{K}'' =\mathcal{K}$ and
\begin{equation}
\label{dimK dimK'}
\dim(\mathcal{K}) \dim(\mathcal{K}') =\dim(\C).
\end{equation}
The subcategory $\mathcal{K}$ is symmetric if and only if
$\mathcal{K} \subseteq \mathcal{K}'$. It is modular if and only if
$\mathcal{K} \cap \mathcal{K}' = \Vec$, in which case $\mathcal{K}'$
is also modular and there is a braided equivalence $\C \cong \mathcal{K} \boxtimes \mathcal{K}'$.

Let $\C$ be a modular category. Then by \cite[Corollary 6.9]{GN},
$\C_{pt} = (\C_{ad})'$.


\subsection{Gauss sums and central charge in modular categories}\label{gauss}

Let $\C$ be a modular category.
For any subcategory $\K$ of $\C$
the {\em Gauss sums} of  $\K$ are defined by
\begin{equation}
\label{Gauss sums}
\tau^{\pm}(\K) = \sum_{X\in \O(\K)}\, \theta_X^{\pm 1}  d(X)^2.
\end{equation}

Below we summarize some basic properties of twists and Gauss sums
(see e.g., \cite[Section 3.1]{BK} for proofs).

Each $\theta_X, \, X\in \O(\C),$ is a root of unity (this
statement is known as Vafa's theorem). The Gauss sums are
multiplicative with respect to tensor product of modular
categories, i.e.,  if $\C_1, \C_2$ are modular categories then
\begin{equation}
\label{multiplicativity of tau}
\tau^\pm(\C_1 \bt \C_2)  = \tau^\pm(\C_1) \tau^\pm(\C_2).
\end{equation}
We also have that
\begin{equation}
\label{mod tau =dimC}
\tau^+(\C)\tau^-(\C) =\dim(\C).
\end{equation}

When $k=\mathbb{C}$ the {\em multiplicative central charge}
$\xi(\C)$ is defined by
\begin{equation}
\label{central charge}
\xi(\C) = \frac{\tau^+(\C)}{\sqrt{\dim(\C)}},
\end{equation}
where $\sqrt{\dim(\C)}$ is the positive root.
If $\dim(\C)$ is a square of an integer,
then Formula \eqref{central charge} makes sense even
if $k \neq \mathbb{C}$.
By Vafa's theorem, $\xi(\C)$ is a root of unity.

\begin{example}
The center $\Z(\C)$ of any fusion category $\C$ (see Example
\ref{center}) is a modular category with central charge $1$
\cite[Theorem 1.2]{Mu4}.
\end{example}

\subsection{Symmetric fusion categories}\label{Deligne}
The structure of symmetric
fusion categories is known, thanks to Deligne's work \cite{De}.
Namely, let $G$ be a finite group and let $z\in G$
be a central element such that $z^2=1$. Consider the category $Rep(G)$ with its standard
symmetric braiding $\sigma_{X,Y}$. Then the map
$\sigma'_{X,Y}=\frac12(1+z|_X+z|_Y-z|_Xz|_Y)\sigma_{X,Y}$ is also a
symmetric braiding on the category $Rep(G)$ (the
meaning of the
factor $\frac12(1+z|_X+z|_Y-z|_Xz|_Y)$ is the following: if $z|_X$ or $z|_Y$ equals 1,
then this factor is 1; if $z|_X=z|_Y=-1$ then this factor is $(-1)$). We will denote by $Rep(G,z)$
the category $Rep(G)$ with the commutativity constraint defined above.

\begin{theorem} (\cite{De}) Any symmetric fusion category is equivalent (as a braided tensor category)
to $Rep(G,z)$ for uniquely defined $G$ and $z$. The categorical dimension of $X\in Rep(G,z)$
equals $Tr(z|_X)$ and $\dim(\C)=\FPdim(\C)=|G|$.
\end{theorem}

Now assume that the category $Rep(G,z)$ is endowed with a twist $\theta$ such that the dimension
of any object is non-negative. It follows immediately from the theorem that $\theta_X=z|_X$.
We have

\begin{corollary}
\label{deligne} Let $\C$ be a symmetric fusion category with the
canonical spherical structure (see \ref{note: sph when k=C}).
\begin{enumerate}
\item[(i)] If $\dim(\C)$ is odd then $\theta_X=\id_X$ for any $X\in \C$.
\item[(ii)] In general either $\theta_X=\id_X$ for any $X\in \C$, or $\C$
contains a fusion subcategory $\C_1\subset \C$ such that $\FPdim(\C_1)=
\frac12\FPdim(\C)$ and $\theta_X=\id_X$ for any $X\in \C_1$.
\end{enumerate}
\end{corollary}
\begin{proof} As for (i), it is clear that $z=1$. For (ii) one takes $\C_1=Rep(G/\langle z\rangle)
\subset \Rep(G)$.
\end{proof}

\section{Isotropic subcategories and Brugui\`{e}res-M\"uger modularization}\label{isotsub}
\label{Sect3}

\subsection{Modularization}
\label{subsect modularization}

\begin{definition}
\label{def:isotropic} Let $\C$ be a premodular category with
braiding $c$ and twist $\theta$. A fusion subcategory  $\E$ of
$\C$ is called {\em isotropic} if $\theta$ restricts to the
identity on $\E$, i.e., if $\theta_X =\id_X$ for all $X\in \E$. An
isotropic subcategory $\E$ is called {\em Lagrangian} if
$\mathcal{E}= \mathcal{E}'$. The category $\C$ is called {\em
hyperbolic} if it has a Lagrangian subcategory and
{\em anisotropic} if it has no non-trivial isotropic
subcategories.

\end{definition}

\begin{remark}\label{rmk}
\begin{enumerate}
\item[(a)] When $\C =\C(G,\, q)$ is a pointed modular category
defined in Example~\ref{CGQ} then isotropic and Lagrangian
subcategories of $\C$ correspond to isotropic and Lagrangian
subgroups of $(G,\, q)$, respectively. We discuss properties of
pointed modular categories in Section~\ref{modpcat}.
\item[(b)] Let $G$ be a finite group and let $\omega \in Z^3(G,\, k^{\times})$.
Consider the pointed fusion category $\Vec_G^\omega$. Its center
$\C = \Z(\Vec_G^\omega)$ is a modular category.  It contains a
Lagrangian subcategory $\E \cong \Rep(G)$ formed  by all objects
in $\C$ which are sent to multiples of the unit object of
$\Vec_G^\omega$ by the forgetful functor $\Z(\Vec_G^\omega) \to
\Vec_G^\omega$.
\item[(c)]  It follows from the balancing axiom~\eqref{balancing ax}
that an  isotropic subcategory $\E \subseteq \C$ is always
symmetric. Conversely, if $\E$ is symmetric and $\dim(\E)$ is odd
then $\E$ is isotropic, see \ref{Deligne}. In particular, if $\dim(\C)$ is odd then
any symmetric subcategory of $\C$ is isotropic.
\item[(d)] Recall that we assume that $\C$ is endowed with a canonical spherical
structure, see \ref{note: sph when k=C}. Any isotropic subcategory
$\E \subset \C$ is equivalent, as a symmetric  category, to
$\Rep(G)$ for a canonically defined group $G$ with  its standard
braiding and identical twist, see \ref{Deligne}. In particular, if
$\E$ is Lagrangian then $\dim(\C)=\dim(\E)^2$ is a square of an
integer.
\end{enumerate}
\end{remark}

Let $\C$ be a premodular category such that its centralizer
$\C'$ is isotropic and dimensions of all objects $X\in \C'$ are non-negative.
Let us recall a construction, due
to A.~ Brugui\`{e}res \cite{Br} and M.~M\"uger \cite{Mu1}, which
associates to $\C$ a modular category  $\bar \C$ and a surjective
braided tensor functor $\C \to \bar \C$.

%

Let  $G(\C)$ be the unique (up to an isomorphism) group such that
the category $\C'$ is equivalent, as a premodular category, to
$\Rep(G(\C))$ with  its standard symmetric braiding and identity
twist.

Let $A$ be the algebra of functions on $G(\C)$. The group $G(\C)$
acts on $A$ via left translations and so $A$ is a commutative
algebra in $\C'$ and hence in $\C$.

Consider the category $\bar \C := \C_A$ of right $A$-modules in the
category $\C$ (see, e.g., \cite[1.2]{KiO}).  It was shown in
\cite{Br, KiO, Mu1} that $\bar \C$ is a braided fusion category and
that the ``free module'' functor
\begin{equation}
\label{free module functor} F: \C \to \bar \C,\, X\mapsto X\otimes
A
\end{equation}
is surjective and has a canonical structure
of a braided tensor functor. One can define a twist $\phi$ on $\bar\C$
in such a way that $\phi_Y =\theta_X$
for all $Y\in \O(\bar\C)$ and $X\in \O(\C)$ for which $\Hom_\C(X, Y)\neq 0$.
It follows that the category $\bar\C$ is modular, see \cite{Br, Mu1, KiO} for details.
We will call the category $\bar\C$ a {\em modularization} of $\C$.

Let $d$ and $\bar{d}$ denote the dimension functions in $\C$ and
$\bar\C$, respectively.
For any object $X$ in $\bar\C$ one has
\begin{equation}
\label{dims}
\bar{d}(X) = \frac{d(X)}{d(A)},
\end{equation}
cf.\ \cite[Theorem 3.5]{KiO}, \cite[Proposition 3.7]{Br}.

\begin{remark}\label{newrmk}
Let $\E$ be an isotropic subcategory of a modular category $\C$.
Then $\dim(\bar{\E'}) = \dim(\C)/\dim(\E)^2$ (see e.g.
\cite{KiO}).
\end{remark}

\subsection{Invariance of the central charge}

In this subsection we prove an invariance property of the central
charge with respect to modularization, which will be crucial in
the sequel.

\begin{theorem}
\label{tau after mod}
Let $\C$ be a modular category and let $\E$ be an isotropic
subcategory of $\C$. Let $F: \E' \to \bar{\E'}$ be the canonical
braided tensor functor from $\E'$
to its modularization. Then $\xi(\bar{\E'}) = \xi(\C)$.
\end{theorem}
\begin{proof}
Let $A$ be the canonical commutative algebra in $\E$. We have
$\dim(\E)= d(A)$. By definition,   $\bar{\E'}$ is the category of
left $A$-modules in $\E'$.

Let us compute the Gauss sums of  $\bar{\E'}$:
\begin{eqnarray*}
\lefteqn{ \dim(\E)\tau^\pm(\bar{\E'}) =  \dim(\E)\sum_{Y\in\O(\bar{\E'})}\,\phi^{\pm 1}_Y \bar{d}(Y)^2=}\\
&=&  \sum_{Y\in\O(\bar{\E'})}\, \phi^{\pm 1}_Y  d(Y) \bar{d}(Y) \\
&=& \sum_{Y\in\O(\bar{\E'})}\, \phi^{\pm 1}_Y  \left( \sum_{X\in\O(\C)}\,
     \dim_k\Hom_\C(X,\, Y) d(X) \right) \bar{d}(Y) \\
&=& \sum_{X\in\O(\C)}\, \theta^{\pm 1}_X d(X)
    \left( \sum_{Y\in\O(\bar{\E'})}\,
    \dim_k\Hom_\C(X,\, Y) \bar{d}(Y) \right) \\
&=&  \sum_{X\in\O(\C)}\, \theta^{\pm 1}_X d(X)
     \left( \sum_{Y\in\O(\bar{\E'})}\,
     \dim_k\Hom_{\bar{\E'}}(X \ot A,\, Y) \bar{d}(Y) \right) \\
&=&  \sum_{X\in\O(\C)}\, \theta^{\pm 1}_X d(X)
     \bar{d} (F(X)) = \tau^\pm(\C), 
\end{eqnarray*}
where we used the relation \eqref{dims}
and the fact that $F$ is an  adjoint of the forgetful functor from $\bar{\E'}$ to $\E'$.


Combining this with the equation $\dim(\bar{\E'}) =
\dim(\C)/\dim(\E)^2$ (see Remark~\ref{newrmk}) we obtain the
result.
\end{proof}

\subsection{Maximal isotropic subcategories}

Let $\C$ be a modular category and let  $\L$ be an isotropic
subcategory of $\C$ which is maximal among isotropic subcategories
of~$\C$. Below we will show that the braided equivalence class of the modular
category  $\bar{\L'}$ (the modularization of $\L'$ by $\L$) is
independent  of the choice of $\L$.

Let $\C$ be a fusion category and let $\A$ and $\B$ be its fusion
subcategories such that $X \ot Y \cong Y\ot X$ for all $X\in\O(\A)$
and $Y\in \O(\B)$.
Let $\A\vee \B$ denote the fusion subcategory of $\C$  generated
by $\A$ and $\B$, i.e.,
consisting of   all subobjects of $X\ot Y$, where $X\in \O(\A)$
and $Y \in \O(\B)$. Recall that the regular element of
$K_0(\C)\otimes_\mathbb{Z} \mathbb{C}$ is
$R_\C =\sum_{X\in \O(\C)}\, d(X) X$.
It is defined up to a scalar multiple  by the property that
$Y \otimes R_\C  = d(Y) R_\C$ for all $Y\in\O(\C)$ \cite{ENO}.

\begin{lemma}
\label{diamond dimensions}
Let $\C,\,\A,\,\B$ be as above. Then
$
\dim(\A \vee \B) =\frac{\dim(\A)\dim(\B)}{\dim(\A\cap\B)}.
$
\end{lemma}
\begin{proof}
It is easy to see that
\begin{equation}
\label{diamond} R_\A \otimes R_\B =a R_{\A\vee \B},
\end{equation}
where the scalar $a$ is equal to the multiplicity of the unit
object $\mathbf{1}$ in $R_\A \ot R_\B$, which is the same as the
multiplicity of $\mathbf{1}$ in $\sum_{Z\in \O(\A\cap\B)}\, d(Z)^2
Z \otimes Z^*$. Hence, $a = \dim(\A\cap\B)$. Taking dimensions of
both sides of \eqref{diamond} we get the result.
\end{proof}

Let $L(\C)$  denote the lattice of fusion subcategories of a
fusion category $\C$. For any two subcategories $\mathcal{A}$ and
$\mathcal{B}$ their meet is their intersection and their joint is
the category $\mathcal{A}\vee \mathcal{B}$.

\begin{lemma}
\label{Dedekind}
Let $\C$ be a fusion category such that $X \ot Y \cong Y\ot X$ for
all objects $X,Y$ in $\C$. For all   $\mathcal{A},\, \mathcal{B},\, \D\in L(\C)$
such that  $\D\subseteq\mathcal{A}$ the following {\em modular law} holds true:
\begin{equation}
\label{modular law}
\mathcal{A} \cap (\mathcal{B} \vee \D)  = (\mathcal{A} \cap \mathcal{B}) \vee \D.
\end{equation}
\end{lemma}
\begin{proof}
A classical theorem of Dedekind in lattice theory states that
\eqref{modular law} is equivalent to the following statement: for
all $\mathcal{A},\, \mathcal{B},\, \D\in L(\C)$ such that
$\D\subseteq\mathcal{A}$, if $\mathcal{A}\cap \mathcal{B} = \D
\cap \mathcal{B}$ and $\mathcal{A}\vee \mathcal{B} = \D \vee
\mathcal{B}$ then $\mathcal{A} =\D$ (see e.g., \cite{MMT}).

Let us prove the latter property.
Take a simple  object  $X\in\mathcal{A}$. Then $X \in \mathcal{A}\vee \mathcal{B} =
\D \vee \mathcal{B}$  so there are simple objects $D\in \D$
and $B\in \mathcal{B}$ such that
$X$ is contained in  $D \ot B$. Therefore, $B$ is contained in $D^*\ot X$
and so $B\in \mathcal{A}$. So $B\in \mathcal{A}\cap \mathcal{B} =
\mathcal{D}\cap \mathcal{B} \subseteq \D$. Hence  $X\in \D$, as required.
\end{proof}

\begin{remark}
When $\C =\Rep(G)$ is the representation category of a finite group $G$,
Lemma~\ref{Dedekind} gives a well-known property of the lattice of
normal subgroups of $G$.
\end{remark}

The next lemma gives an analogue of a diamond isomorphism for the
``quotients by isotropic subcategories.''

\begin{lemma}
\label{diamond isomorphism} Let $\C$ be a modular category, let
$\D$ be an isotropic subcategory of $\C$ and let $\mathcal{B}$  be
a subcategory of $\D'$. Let $A$, $A_0$ be the canonical
commutative algebras in $\D$ and $\D \cap \B$, respectively.

Then the category $\B_{A_0}$ of
$A_0$-modules in $\B$ and  the category
$(\D\vee\B)_A$ of $A$-modules in $\D\vee\B$ are equivalent as braided
tensor categories.
\end{lemma}
\begin{proof}
Note that
\[
\dim(\B_{A_0}) = \frac{\dim(\B)}{\dim(\D \cap \B)}
=\frac{\dim(\D\vee\B)}{\dim(\D)} = \dim((\D\vee\B)_A)
\]
by Lemma~\ref{diamond dimensions}.

Define a  functor $H: \B_{A_0} \to (\D\vee\B)_A$ by $H(X) = X
\otimes_{A_0} A ,\, X \in \B_{A_0}$. Then $H$ has a natural
structure of a braided tensor functor. Note that for $X = Y\ot
A_0,\, Y\in \B$ we have $H(X) = Y\ot A $, i.e., the composition of
$H$ with the free $A_0$-module functor is the free $A$-module
functor. The latter functor is surjective and, hence, so is $H$.

Since a surjective functor between categories of equal
dimension is necessarily an equivalence (see \cite[5.7]{ENO}
or \cite[Proposition 2.20]{EO}) the result follows.
\end{proof}

\begin{proposition}
\label{canonical modularization} Let $C$ be a modular category and
let $\L_1$, $\L_2$ be maximal among isotropic subcategories of
$\mathcal{C}$. Then the modularization $\bar{\L_1'}$ and
$\bar{\L_2'}$ are equivalent as braided fusion categories.
\end{proposition}
\begin{proof}
Let $\D =\L_1$ and  $\B = \L_1' \cap \L_2'$.
By maximality of $\L_1, \L_2$ we have  $\L_1'\cap \L_2\subseteq \L_1$
and $\L_1\cap \L_2' \subseteq \L_2$.
Therefore,   $\D \cap \B =  \L_1\cap \L_2$  and
$\D \vee \B = \L_1' \cap (\L_1 \vee \L_2')= \L_1'$
by Lemma~\ref{Dedekind}.

Let $A_0$ be the canonical commutative algebra in $\L_1 \cap
\L_2$. Applying Lemma~\ref{diamond isomorphism} we see that
$\bar{\L_1'}$ is equivalent to the category $(\L_1' \cap
\L_2')_{A_0}$ of $A_0$-modules in $\L_1' \cap \L_2'$.  The
proposition now follows by interchanging $\L_1$ and $\L_2$.
\end{proof}

\begin{remark}
\begin{enumerate}
\item[(i)]
We can call the modular category $\bar{\L_1'}$ constructed
in  the proof of Proposition~\ref{canonical modularization}
``the'' {\em canonical  modularization} corresponding to $\C$
(it measures the failure of $\C$ to be hyperbolic).
The above proof gives
a concrete equivalence $\bar{\L_1'} \cong \bar{\L_2'}$. But
given another  maximal  isotropic subcategory $\L_3\subset  \C$
the composition of equivalences $\bar{\L_1'} \cong \bar{\L_2'}$ and
$\bar{\L_2'} \cong \bar{\L_3'}$ is {\em not} in general equal to
the equivalence $\bar{\L_1'} \cong \bar{\L_3'}$. This is why we
put ``the'' above in quotation marks.
\item[(ii)]
For a maximal isotropic subcategory $\L \subset \C$ the corresponding
modularization does {\em not} have to be anisotropic, in contrast
with the situation for metric groups. Examples illustrating this
phenomenon are, e.g.,   the centers of non-group theoretical Tambara-Yamagami
categories considered in \cite[Remark 8.48]{ENO}.
\end{enumerate}
\end{remark}


\section{Reconstruction of a twisted group double from a
Lagrangian subcategory}

\subsection{$\C$-algebras}
Let us recall the following definition
from \cite{KiO}.

\begin{definition}
\label{C-algebra} Let $\C$ be a ribbon fusion category. A  {\em
$\C-$algebra} is a commutative algebra $A$ in $\C$ such that $\dim
\Hom(\be ,A)=1$, the pairing $A\otimes A\to A\to \be$ given by the
multiplication of $A$ is non-degenerate, $\theta_A=\id_A$ and
$\dim(A)\ne 0$.
\end{definition}

Let $\C$ be a modular category, let $A$ be a $\C-$algebra, and let
$\C_A$ be the fusion category of right $A-$modules with the tensor
product $\ot_A$. The free module functor $F: \C \to \C_A,\, X
\mapsto X\ot A$ has an obvious structure of a {\em central
functor}. By this we mean that there is a natural family of
isomorphisms $F(X) \ot_A Y \cong Y \ot_A F(X),\, X\in \C,\, Y\in
\C_A$, satisfying an obvious multiplication compatibility, see
e.g. \cite[2.1]{Be}. Indeed, we have $F(X)=X\otimes A$, and hence
$F(X)\otimes_A Y = X\otimes Y$. Similarly, $Y\otimes_A
F(X)=Y\otimes X$. These two objects are isomorphic via the
braiding of $\C$ (one can  check that the braiding gives an
isomorphism of $A$-modules using the commutativity of $A$).

Thus, the functor $F$ extends to a functor $\tilde F: \C \to
\Z(\C_A)$ in such a way that  $F$ is the composition of
$\tilde{F}$ and the forgetful functor $\Z(\C_A) \to \C_A$.

\begin{proposition}
\label{tilde F}
The functor $\tilde F:\C \to \Z(\C_A)$ is injective
(that is fully faithful).
\end{proposition}
\begin{proof} Consider $\C_A$ as a module category over $\C$
via $F$ and over $\Z(\C_A)$ via $\tilde{F}$. We will prove the
dual statement (see \cite[Proposition 5.3]{ENO}), namely that the
functor $T:\C_A\boxtimes \C_A^{op}\to \C^*_{\C_A}$ dual  to
$\tilde F$ is surjective (here and below the superscript $op$
refers to the tensor category with the opposite tensor product).
Recall (see e.g. \cite{O1}) that the category $\C^*_{\C_A}$ is
identified with the category of $A-$bimodules. An explicit
description of the functor $T$ is the following: by definition,
any $M\in \C_A$ is a right $A-$module. Using the braiding and its
inverse one can define on $M$ two structures of a left $A-$module:
$A\otimes M\stackrel{c_{A,M}^{\pm 1}}{\longrightarrow}M\otimes
A\to M$. Both structures make $M$ into an $A-$bimodule, and we
will denote the two results by $M_+$ and $M_-$, respectively. Then
we have $T(M\boxtimes N)= M_+\otimes_AN_-$.  In particular we see
that the functor $\C \boxtimes \C^{op}\stackrel{F\boxtimes
F}{\longrightarrow} \C_A\boxtimes
\C_A^{op}\stackrel{T}{\longrightarrow}\C^*_{\C_A}$ coincides with
the functor $\C\boxtimes \C^{op}\simeq \Z(\C)\simeq
\Z(\C^*_{\C_A})\to \C^*_{\C_A}$ (see \cite{O2}). Since the functor
$\Z(\C^*_{\C_A})\to \C^*_{\C_A}$ is surjective (see
\cite[3.39]{EO}) we see that the functor $T$ is surjective. The
proposition is proved.
\end{proof}

\begin{remark} Note that since $\C$ and $\Z(\C_A)$ are modular we have a
factorization $\Z(\C_A)=\C\boxtimes \mathcal{D}$, where
$\mathcal{D}$ is the centralizer of $\C$ in $\Z(\C_A)$. One
observes that $\mathcal{D}$ is identified with the category of
``dyslectic" $A-$modules $\Rep^0(A)$, see \cite{KiO,P}.
\end{remark}

\begin{corollary} \label{equ}
Assume that $\dim(A)=\sqrt{\dim(\C)}$. Then the functors $\tilde
F: \C \to \Z(\C_A)$ and $T: \C_A\boxtimes \C_A^{op}\to \C^*_{\C_A}$
are tensor equivalences.
\end{corollary}
\begin{proof} We have already seen that $\dim(\C_A)=\frac{\dim(\C)}{\dim(A)}$.
Hence,  $\dim(\Z(\C_A))=\frac{\dim(\C)^2}{\dim(A)^2}=\dim(\C)$. Since
$\tilde F$ is an  injective functor between categories
of equal dimension,  it is necessarily  an equivalence by
\cite[Proposition 2.19]{EO}. Hence the dual
functor $T$ is also an equivalence.
\end{proof}

\subsection{Hyperbolic modular categories as twisted group doubles}
We are now ready to state and prove our first main result which
relates hyperbolic modular categories and twisted doubles of
finite groups.

Let $\C$ be a modular category. Consider the set of all triples
$(G, \omega, F)$, where $G$ is a finite group, $\omega\in Z^3(G,
k^\times)$, and $F: \C \xrightarrow{\sim} \Z(\Vec_G^\omega)$ is a
braided tensor equivalence. Let us say that two triples  $(G_1,
\omega_1, F_1)$ and $(G_2, \omega_2, F_2)$ are equivalent if there
exists a tensor  equivalence $\iota: \Vec_{G_1}^{\omega_1}
\xrightarrow{\sim} \Vec_{G_2}^{\omega_2}$ such that
$\mathcal{F}_2\circ F_2  = \iota\circ \mathcal{F}_1\circ F_1$,
where $\mathcal{F}_i : \Z(\Vec_{G_i}^{\omega_i}) \to
\Vec_{G_i}^{\omega_i}, \, i=1,2$, are the canonical forgetful
functors.

Let $\mbox{E}(\C)$ be the set of all equivalences classes of
triples $(G, \omega, F)$. Let $\mbox{Lagr}(\C)$  be the set of all
Lagrangian subcategories of $\C$.

\begin{theorem}\label{main1}
For any modular category $\C$
there is a natural bijection
\[
f: \mbox{E}(\C)  \xrightarrow{\sim}  \mbox{Lagr}(\C).
\]
\end{theorem}
\begin{proof}
The map $f$ is defined as follows. Note that each braided tensor
equivalence $F: \C\xrightarrow{\sim} \Z(\Vec_G^\omega)$ gives rise to the
Lagrangian subcategory $f(G, \omega, F)$ of  $\C$ formed by all objects sent to
multiples of the unit object $\be$ under the forgetful functor
$\Z(\Vec_G^\omega)\to \Vec_G^\omega$. This subcategory is clearly
the same for all equivalent choices of $(G, \omega, F)$.

Conversely, given a Lagrangian subcategory $\E \subseteq \C$
it follows from Deligne's theorem \cite{De} that
$\mathcal E = \Rep(G)$ for a unique (up to isomorphism) finite
group $G$.
Let $A = \Fun(G)\in  \Rep(G)=\E \subset \C$. It is clear that $A$
is a $\C-$algebra and $\dim(A)=\dim(\E)=\sqrt{\dim(\C)}$. Then by
Corollary \ref{equ}, the functor $\tilde F:\C \to \Z(\C_A)$ is an
equivalence.

Finally, let us show that $\C_A$ is pointed and $K_0(\C_A) =
\mathbb{Z}G$. Note that there are $|G|$ non-isomorphic structures
$A_g$, $g\in G$, of an invertible $A$-bimodule on $A$, since the
category of $A$-bimodules in $\mathcal E$ is equivalent to
$\Vec_G$. For each $A_g$ there is a pair $X,Y$ of simple objects
in $\C_A$ such that $T(X\boxtimes Y) = A_g$. Taking the forgetful
functor to $\C_A$ we obtain $Y = X^*$  and $X$ is invertible.
Hence, for each $g\in G$ there is a unique invertible $X_g\in
\C_A$ such that $T(X_g \boxtimes X_g^*) = A_g$, and therefore $g
\mapsto X_g$ is an isomorphism of $K_0$ rings. Thus, $\C_A \cong
\Vec_G^\omega$ for some $\omega \in \Z^3(G, k^\times)$. We set
$h(\E)$ to be the class of the equivalence  $\tilde{F}:  \C
\xrightarrow{\sim} \Z(\C_A)$.

Let show that the above constructions $f$ and $h$ are inverses of each other.
Let $\E$ be a Lagrangian subcategory of $\C$ and let $A$ be the algebra defined
in the previous paragraph.  The forgetful functor from $\C\cong \Z(\C_A)$
to $\C_A$ is the free module functor, and so $f(h(\E))$
consists of all objects  $X$ in $\C$ such that $X \otimes A$ is a multiple of $A$.
Since $A$ is the regular object of  $\E$, it follows that
$f(h(\E))= \E$ and $f\circ h =\id$.

Proving that $h \circ f=\id$ amounts to a verification of the
following fact. Let $G$ be a finite group, let $\omega\in Z^3(G,\,
k^\times)$, and let $A =\Fun(G)$ be the canonical algebra in
$\Rep(G)\subset \Z(\Vec_G^\omega)$. Then the category of
$A$-modules in $\Z(\Vec_G^\omega)$ is equivalent to
$\Vec_G^\omega$ and the functor of taking the free $A$-module
coincides with the forgetful functor from  $\Z(\Vec_G^\omega)$ to
$\Vec_G^\omega$. This is straightforward and is left to the
reader.
\end{proof}

\begin{remark}
Our reconstruction of the representation category of a twisted
group double from a Lagrangian subcategory can be viewed as a
categorical analogue of the following reconstruction of the double
of a quasi-Lie bialgebra from a Manin pair (i.e., a pair
consisting of a metric Lie algebra and its Lagrangian subalgebra)
in the theory of quantum groups \cite[Section 2]{Dr}.

Let $\mathfrak{g}$ be a finite-dimensional metric Lie algebra
(i.e., a Lie algebra on which a nondegenerate invariant symmetric
bilinear form is given). Let $\mathfrak{l}$ be a Lagrangian
subalgebra of $\mathfrak{g}$. Then $\mathfrak{l}$ has  a structure
of a quasi-Lie bialgebra and there is an isomorphism  between
$\mathfrak{g}$ and the double $\mathfrak{D}(\mathfrak{l})$ of
$\mathfrak{l}$. The correspondence between Lagrangian subalgebras
of   $\mathfrak{g}$ and  doubles isomorphic to $\mathfrak{g}$ is
bijective, see \cite[Section 2]{Dr} for details.
\end{remark}

\begin{remark}
Given a hyperbolic modular category $\C$ there is no canonical way
to assign to it a pair $(G,\, \omega)$ such that $\C \cong
\Z(\Vec_G^\omega)$ as a braided fusion category. Indeed, it
follows from \cite{EG1} that there exist non-isomorphic finite
groups $G_1, G_2$ such that $\Z(\Vec_{G_1}) \cong \Z(\Vec_{G_2})$
as braided fusion categories. (See also \cite{N}.)
\end{remark}

\begin{theorem}
\label{dichotomy} Let $\C$ be a modular category such that
$\dim(C)=n^2,\, n\in \mathbb{Z}^+$,  and such that $\xi(\C)=1$.
Assume that $\C$ contains a symmetric subcategory $\V$ such that
$\dim(\V)=n$. Then either $\C$ is the center of a pointed category
or it contains an object with non-integer dimension.
\end{theorem}

\begin{proof} Assume that $\V$ is not isotropic.
Then $\V$ contains an isotropic subcategory $\K$ such that
$\dim(\K)=\frac12 \dim(\V)$ (this follows from Deligne's
description of symmetric categories, see 2.7). Hence the category
$\bar \K'$ (modularization of $\K'$) has dimension 4 and central
charge 1. It follows from the explicit classification given in
Example 5.1 (b),(d) that the category $\bar \K'$ contains an
isotropic subcategory of dimension 2; clearly this subcategory is
equivalent to $\Rep(\BZ/2\BZ)$. Let $A_1=\Fun(\BZ/2\BZ)$ be the
commutative algebra of dimension 2 in this subcategory. Let $I:
\bar \K'\to \K'$ be the right adjoint functor to the
modularization functor $F: \K'\to \bar \K'$.

We claim that the object $A:=I(A_1)$ has a canonical structure of
a $\C-$algebra. Indeed, we have a canonical morphism in
$\Hom(F(A),A_1)=\Hom(A,I(A_1))=\Hom(A,A)\ni \id$. Using this one
can construct a multiplication on $A$ via $\Hom(A_1\otimes
A_1,A_1)\to \Hom(F(A)\otimes F(A),A_1)=\Hom(F(A\otimes A),A_1)=
\Hom(A\otimes A,A)$. Since the functor $F$ is braided it follows
from the commutativity of $A_1$ that $A$ is commutative. Other
conditions from Definition~\ref{C-algebra} are also easy to
check. In particular
$\dim(A)=\dim(\K)\dim(A_1)=\dim(\V)=\sqrt{\dim(\C)}$. We also
note that the category $Rep_{\K'}(A)$ contains precisely two
simple objects (actually, the functor $M\mapsto I(M)$ is an
equivalence of categories between $\Rep_{\bar \K'}(A_1)$ and
$\Rep_{\K'}(A)$); we will call these two objects $\be$ (for $A$
itself considered as an $A-$module) and $\delta$. Clearly $\delta
\otimes_A\delta =\be$.

By Corollary~\ref{equ}, we have an equivalence
$\C_{\C_A}^*=\C_A\boxtimes \C_A^{op}$. Moreover, the forgetful
functor $\C_{\C_A}^*\to \C_A$ corresponds to the tensor product
functor $\C_A\boxtimes \C_A^{op}\to \C_A$. Now consider the
subcategory $(\K')^*_{\C_A}\subset \C^*_{\C_A}$ (in other words
$A-$bimodules in $\K'$); the forgetful functor above restricts to
$S:(\K')^*_{\C_A}\to \Rep_{\K'}(A)$.

Let $M\in (\K')^*_{\C_A}$ be a simple object. We claim that there
are three possibilities: 1) $S(M)=\be$, 2) $S(M)=\delta$ or 3)
$S(M)=\be \oplus \delta$. Indeed, $M=X\boxtimes Y\in \C_A\boxtimes
\C_A^{op}$ and $S(M)=X\otimes Y$ for some simple $X,Y\in \C_A$.
Since $\be$ and $\delta$ are invertible the result is clear.

Now, notice that if there exists $M$ as in case 3) then we have $X=Y^*$ and
$\dim(X)=\dim(Y)=\sqrt{2}$. Thus the category $\C_A$ contains an object
with non-integer dimension, which implies that the category $\C$ contains an
object with non-integer dimension (see e.g. \cite[Corollary 8.36]{ENO}), and
the theorem is proved in this case.
Hence we will assume that for any $M\in (\K')^*_{\C_A}$ only 1) or 2) holds. This
implies that all objects of $(\K')^*_{\C_A}$ are invertible. Note
that $\dim((\K')^*_{\C_A})=\dim(K')=2\sqrt{\dim(\C)}$ and hence we
have precisely $2\sqrt{\dim(\C)}$ simple objects. Consider all
objects $M\in (\K')^*_{\C_A}$ such that $S(M)=\be$; it is easy to
see that there are precisely $\sqrt{\dim(C)}$ of those (indeed,
$X\boxtimes Y\mapsto X\boxtimes (Y\otimes_A\delta)$ gives a
bijection between simple bimodules $M$ with $S(M)=\be$ and simple
bimodules $M$ with $S(M)=\delta$). Let $G$ be the group of
isomorphism classes of all objects $M\in (\K')^*_{\C_A}$ with
$S(M)=\be$ (thus $|G|=\sqrt{\dim(\C)}$). Any object of this type
is of the form $X_g\boxtimes (X_g)^*$ for some invertible $X_g\in
\C_A$. Thus we already constructed $\sqrt{\dim(\C)}$ invertible
simple objects in $\C_A$. Since $\dim(\C_A)=\sqrt{\dim(\C)}$ the
objects $X_g$ exhaust all simple objects in $\C_A$. By
Corollary~\ref{equ}, we are done.
\end{proof}

\subsection{A criterion for a modular category to be group-theoretical}
 
Let $\C$ be a modular category. It is known that the entries
of the $S$-matrix of $\C$ are cyclotomic integers \cite{CG, dBG}.
Hence, we may identify them with complex numbers. In particular,
the notions of complex conjugation and absolute value of
the elements of the $S$-matrix make sense.

\begin{remark}
\label{recall Kad'}
Let $\K \subseteq \C$
be a fusion subcategory. Recall from \cite{GN} that $(\K_{ad})'$
is spanned by simple objects $Y$ such that $|s_{XY}|=d_X d_Y$
for all simple $X$ in $\K$. In this case the ratio 
$b(X, Y):= s_{XY}/(d_Xd_Y)$ is a root of unity.  
Furthermore, for all simple $X\in \K,\, Y_1, Y_2 \in \K_{ad}'$ 
and any simple subobject $Z$  of $Y_1\ot Y_2$ we have
\begin{equation}
\label{b mult}
b(X, Y_1) b(X, Y_2) = b(X, Z),
\end{equation}
as explained in \cite{Mu2}.
\end{remark}

\begin{lemma}
\label{large sym}
Let $\C$ be a modular category and let $\K \subseteq \C$
be a fusion subcategory such that 
$\K \subseteq (\K_{ad})'$.
\begin{enumerate}
\item[(1)] There is a grading $\K = \oplus_{g\in G}\, \K_g$ such that
$\K_1=\K'\cap \K$.
\item[(2)] There is a non-degenerate symmetric 
bilinear form $b$ on $G$ such that
$b(g,h) =  s_{XY}/(d_Xd_Y)$ for all $X\in \K_g$ and $Y\in \K_h$.
\item[(3)] If $\K'\cap \K$ is isotropic then there is a non-degenerate
quadratic form $q$ on $G$ such that $q(g) =\theta_X$ for all 
$X\in \K_g$. In this case $b$ is the  bilinear form corresponding to $q$.
\end{enumerate}
\end{lemma}
\begin{proof}
Since $\K_{ad} \subseteq \K'\cap \K \subseteq \K$ the assertion (1)
follows from \cite{GN}. 

Let $b(X, Y)= s_{XY}/(d_Xd_Y)$
for all simple $X,Y \in \K$. Clearly, $b$ is symmetric
and $b(X, Y)= 1$ for all simple $X$ in $\K$ if and only
if $Y \in \K'\cap \K =\K_1$.
To prove (2) it suffices to check that $b$ depends only
on $h \in G$ such that $Y\in \K_h$ (then the $G$-linear 
property follows from \eqref{b mult}). Let $Y_1, Y_2$
be simple objects in $\K_h$. Then $Y_1\ot Y_2^* \in \K'\cap \K$
and so $b(X, Y_1) b(X, Y_2^*)  = 1$, whence
$b(X, Y_1) =  b(X, Y_2)$, as desired.

Finally, (3) is a direct consequence of our discussion in
Section~\ref{subsect modularization}.
\end{proof}

For a subcategory $\K \subseteq \C$ satisfying the hypothesis
of Lemma~\ref{large sym} let $(G_\K, b_\K)$
be the corresponding abelian grading group and bilinear form.
Note that if such $\K$ is considered as a subcategory of
$\C^{rev}$ then the corresponding bilinear form is $(G_\K, b_\K^{-1})$.

\begin{theorem}
\label{max symmetric  criterion}
Let $\C$ be a modular category. Then symmetric subcategories
of $\Z(\C)\cong \C \bt \C^{\rev}$ of dimension $\dim(\C)$ 
are in bijection with triples $(\L,\,\mathcal{R}, \iota)$,
where $\mathcal{L} \subseteq \C$,
$\mathcal{R} \subseteq \C^{\rev}$
are symmetric subcategories such that
$(\L')_{ad} \subseteq \L$, $(\mathcal{R}')_{ad} \subseteq
\mathcal{R}$, and $\iota: (G_{\L'}, b_{\L'}) 
\cong (G_{\mathcal{R}'}, b_{\mathcal{R}'})$ is an isomorphism
of bilinear forms.

Namely, any such subcategory is of the form
\begin{equation}
\D_{\L,\mathcal{R}, \iota} = 
\oplus_{g\in G_{\L'}} \, \L_g \bt \mathcal{R}_{\iota(g)}.
\end{equation}
\end{theorem}
\begin{proof}
 Let $X_1\bt Y_1$
and  $X_2\bt Y_2$ be two simple objects of $\C \bt \C^{\rev}$.
They centralize each other if and only if
\begin{gather}
\label{|L|}  |s_{X_1 X_2}| = d_{X_1}d_{X_2},\\
\label{|R|} |s_{Y_1 Y_2}| = d_{Y_1}d_{Y_2}, \quad \mbox{ and }\\ 
\label{LR=1} \frac{s_{X_1 X_2}}{d_{X_1}d_{X_2}} 
\frac{s_{Y_1 Y_2}}{d_{Y_1}d_{Y_2}}= 1.
\end{gather}

Let $\D$ be a symmetric subcategory of $\C \bt \C^{\rev}$ and let
$\L$ (respectively, $\mathcal{R}$) be the centralizers of fusion
subcategories of $\C$ (respectively, $\C^{\rev}$)
formed by left (respectively, right)
tensor factors of simple objects in $\D$. By conditions
\eqref{|L|}, \eqref{|R|}, and Remark~\ref{recall Kad'}
we must have ${\L'}_{ad} \subseteq \L$ and
${\mathcal{R}'}_{ad} \subseteq \mathcal{R}$. 
Hence, Lemma~\ref{large sym} gives gradings
$\L' =\oplus_{g \in G_\L}\,(\L')_g$ with $(\L')_1=\L'\cap \L$
and $\mathcal{R}' =\oplus_{g \in G_\mathcal{R}}\,
(\mathcal{R}')_g$ with $(\mathcal{R}')_1=\mathcal{R}'\cap \mathcal{R}$.
The condition \eqref{LR=1} gives an isomorphism
of bilinear forms $\iota: (G_{\L'}, b_{\L'}) 
\cong (G_{\mathcal{R}'}, b_{\mathcal{R}'})$
which is well-defined be the property that whenever
$X\in (\L')_g$ and
$Y\in \mathcal{R}'$ are simple objects such that  $X \bt Y\in \D$
then $Y\in (\mathcal{R}')_{\iota(g)}$.

Note that 
\begin{equation}
\label{D subset}
\D\subseteq 
\oplus_{g\in G_{\L'}} \, \L_g \bt \mathcal{R}_{\iota(g)},
\end{equation}
and hence 
\[
\dim(\D) \leq \dim(L\cap \L')\dim(\mathcal{R}
\cap \mathcal{R}') |G_{\L'}| = \dim(\L') \dim(\mathcal{R}
\cap \mathcal{R}'). 
\]
The same inequality holds with
$\L$ and $\mathcal{R}$ interchanged. Therefore,
\[
\dim(\C)^2 = \dim(\D)^2 
\leq   \dim(\L') \dim(\L'\cap \L) \dim(\mathcal{R}') \dim(\mathcal{R}
\cap \mathcal{R}') \leq  \dim(\C)^2.
\]
Here the first inequality becomes equality if and only if  
the inclusion in \eqref{D subset} is an equality
and the second  inequality becomes equality
if and only if $\L'\cap \L =\L$
and $\mathcal{R}'\cap \mathcal{R} =\mathcal{R}$, i.e., when 
$\L$ and $\mathcal{R}$ are symmetric.
\end{proof}

\begin{remark}
\label{max isotropic criterion}
The subcategory $\D_{\L,\mathcal{R},\iota}$ constructed
in Theorem~\ref{max symmetric  criterion} is Lagrangian
if and only if $\L$ and $\mathcal{R}$ are isotropic subcategories
of $\C$ and $\iota$ is an isomorphism of metric groups.
\end{remark}

\begin{corollary}
\label{group theor criterion}
Let $\C$ be a modular category. The following conditions 
are equivalent:
\begin{enumerate}
\item[(i)] $\C$ is group-theoretical.
\item[(ii)] There is a finite group $G$ and a 3-cocycle
$\omega \in Z^3(G,\, k^\times)$ such that 
$\Z(\C)\cong \Z(Vec_G^\omega)$ as a braided fusion category.
\item[(iii)] $\C \bt \C^{rev}$ contains a Lagrangian subcategory.
\item[(iv)] There is an isotropic subcategory $\E\subset \C$
such that $(\E')_{ad} \subseteq \E$.
\end{enumerate}
\end{corollary}
\begin{proof}
The equivalence (i)$\equivalent$(ii) is a consequence of \cite{ENO},
(ii)$\equivalent$(iii) follows from Theorem~\ref{main1}, and
(iii)$\equivalent$(iv) follows from taking $\E= \mathcal{R}=\L$
and $\iota =\id_{G_{\E'}}$ in 
Theorem~\ref{max symmetric  criterion}, cf.\ 
Remark~\ref{max isotropic criterion}.
\end{proof}

Combining the above criterion with Theorem~\ref{dichotomy}
we obtain the following useful characterization of group-theoretical
modular categories.

\begin{corollary}
\label{useful grthr characterization}
A modular category $\C$ is group-theoretical if and only
if simple objects of $\C$ have integral dimension and
there is a symmetric subcategory $\L \subset \C$ 
such that $(\L')_{ad} \subseteq \L$.
\end{corollary}

\section{Pointed  modular  categories}\label{modpcat}

In this section we analyze the structure of pointed modular
categories, their central charges, and Lagrangian subgroups.
Recall that such categories canonically correspond to metric
groups \cite{Q}.

Let $G=\BZ/n\BZ$. The corresponding braided categories of the form
$\C(\BZ/n\BZ,q)$ are completely classified by numbers $\sigma
=q(1)$ such that $\sigma^n=1$ ($n$ is odd) or $\sigma^{2n}=1$ ($n$
is even). Then the braiding of objects corresponding to $0\le
a,b<n$ is the multiplication by $\sigma^{ab}$ and the twist of the
object $a$ is the multiplication by $\sigma^{a^2}$ (see \cite{Q}).
We will denote the category corresponding to $\sigma$ by
$\C(\BZ/n\BZ,\sigma)$.

\begin{example} \label{peven}
 (a) Let $G=\BZ/2\BZ$. There are 4 possible values of $\sigma$: $\pm 1, \pm i$.
The categories $\C(\BZ/2\BZ,\pm i)$ are modular with central
charge $\frac{1\pm i}{\sqrt{2}}$ and the categories
$\C(\BZ/2\BZ,\pm 1)$ are symmetric. The category $\C(\BZ/2\BZ,1)$
is isotropic and the category $\C(\BZ/2\BZ,-1)$ is not.

(b) Let $G=\BZ/4\BZ$. The twist of the object $2\in \BZ/4\BZ$ is
$\sigma^4=\pm 1$. If this twist is -1 then $\sigma$ is a primitive
8th root of 1 and the corresponding category is modular; its Gauss
sum is $1+\sigma +\sigma^4+\sigma^9=2\sigma$ and the central
charge is $\sigma$. Note that if $\sigma^4=1$ then the category
$\C(\BZ/4\BZ,\sigma)$ contains a nontrivial isotropic subcategory.

(c) Let $G=\BZ/2^k\BZ$ with $k\ge 3$. Since the twist of the
object $2^{k-1}$ is $\sigma^{2^{2k-2}}=1$, the category
$\C(\BZ/2^k\BZ,\sigma)$ always contains a nontrivial isotropic
subcategory.

(d) Let $G=\BZ/2\BZ \times \BZ/2\BZ$. There are five modular
categories with this group. We give for each of them the list of
values of $q$ on nontrivial elements of $G$:
\begin{enumerate}
\item $\C(\BZ/2\BZ \times \BZ/2\BZ,i)$:
the values of $q$ are $i, i, -1$, and the central charge is $i$.
\item $\C(\BZ/2\BZ \times \BZ/2\BZ,-i)$: the values of $q$ are $-i, -i, -1$,
and the central charge is $-i$.
\item $\C(\BZ/2\BZ \times \BZ/2\BZ,-1)$: the values of $q$ are $-1, -1, -1$,
and the central charge is $-1$.
\item $\C(\BZ/2\BZ \times \BZ/2\BZ,1)$: the values of $q$ are $i, -i, 1$,
and the central charge $1$.
\item The double of $\BZ/2\BZ $: the values of $q$ are $1,1,-1$, and
the central charge $1$.
\end{enumerate}
In this list, each category of central charge $1$ contains a
nontrivial isotropic subcategory while the others contain a
nontrivial symmetric (but not isotropic) subcategory.

(e) Let $G=\BZ/2\BZ \times \BZ/4\BZ$. Assume that the category
$\C(G,q)$ does not contain a nontrivial isotropic subcategory.
Then $\C(G,q)$ is equivalent to $\C(\BZ/4\BZ,\sigma)\boxtimes
\C(\BZ/2\BZ,\pm i)$ where $\sigma$ is a primitive $8$th root of 1.
The possible central charges are $\pm 1$ and $\pm i$.

(f) Let $G=\BZ/2\BZ \times \BZ/2\BZ \times \BZ/2\BZ$. Assume that
the category $\C(G,q)$ does not contain a nontrivial isotropic
subcategory. Then $\C(G,q)$ is equivalent to $\C(\BZ/2\BZ \times
\BZ/2\BZ,\sigma)\boxtimes \C(\BZ/2\BZ,\sigma ')$, where
$\sigma'=\pm i$ and $\sigma \ne 1, -\sigma'$.

\end{example}

\begin{example} \label{podd}
Let $p$ be an odd prime.

(a) Let $G=\BZ/p\BZ$. The category $\C(\BZ/p\BZ,\sigma)$ is
modular for $\sigma \ne 1$ and is isotropic for $\sigma =1$. The
central charge of the modular category $\C(\BZ/p\BZ,\sigma)$ is
$\pm 1$ for $p=1\mod{4}$ and $\pm i$ for $p=3\mod{4}$.

(b) Let $G=\BZ/p\BZ \times \BZ/p\BZ$. There are two
modular pointed categories with underlying group $G$.
One has central charge 1 (and is equivalent to the center
of $\BZ/p\BZ$), and the other one has central charge -1.
\end{example}

Recall that for a metric group $(G, q)$ its Gauss sum is
$\tau^\pm(G,\, q) =\sum_{a\in G}\, q(a)^{\pm 1}$. A subgroup $H$
of $G$ is called {\em isotropic} if $q|_H =1$. An isotropic
subgroup is called {\em Lagrangian} if $H^\perp = H$.

The following proposition is well known.

\begin{proposition}
\label{basis} Let $(G,\, q)$ be a non-degenerate metric group such
that $|G|=p^{2n}$ where $p$ is a prime number and $n\in \mathbb{Z}^+$.
Suppose that $\tau^\pm(G, q) = \sqrt{|G|}$ (i.e., the central charge of $G$
is $1$).  Then  $G$ contains a Lagrangian subgroup.
\end{proposition}
\begin{proof}
It suffices to prove that  $G$ contains a non-trivial isotropic
subgroup $H$, then one can pass to $H^\perp/H$ and use induction.

Assume that $p$ is odd. Assume that $G$ contains a direct summand
$\BZ/p^k\BZ$ with $k>1$. Then the subgroup $\BZ/p\BZ \subset
\BZ/p^k\BZ$ is isotropic, since otherwise it is a non-degenerate metric subgroup
of $G$  and hence can be factored. Thus we are reduced to the case
when $G$ is a direct sum of $k$ copies of $\BZ/p\BZ$. When $k
> 2$, the quadratic form on $G$ is
isotropic (by the Chevalley - Waring theorem). Thus we are reduced
to the case $k=2$, which is easy (see Example \ref{podd} (b)).

Assume now that $p=2$. Again assume that $G$ contains a direct
summand $\BZ/2^k\BZ$ with $k>1$. Again the subgroup $\BZ/2\BZ
\subset \BZ/2^k\BZ$ is inside its orthogonal complement; moreover
it is isotropic if $k\ge 3$. If $k=2$ and the subgroup $\BZ/2\BZ
\subset \BZ/4\BZ$ is not isotropic then the subgroup $\BZ/4\BZ$ is
a non-degenerate metric subgroup and hence factors out;
let $G=G_1\oplus \BZ/4\BZ$ be the
corresponding decomposition of $G$. If $G_1$ contains $\BZ/2\BZ$
such that $\BZ/2\BZ \subseteq \BZ/2\BZ^\perp$ then we are done: if
this subgroup is not isotropic then the diagonal subgroup
$\BZ/2\BZ \subset \BZ/2\BZ \oplus \BZ/2\BZ \subset G_1\oplus
\BZ/4\BZ$ is isotropic. Thus $G_1$ is a sum of $\BZ/2\BZ$'s and
each summand is non-degenerate. But note that the central charge
of a non-degenerate metric group $\BZ/4\BZ$ is a primitive eighth root of $1$
(see Example \ref{peven} (b))
which is also the central charge of a non-degenerate metric
$\BZ/2\BZ$ (see Example \ref{peven} (a)). This
implies that the number of $\BZ/2\BZ$ summands in $G_1$ is odd
which is impossible since the order of $G$ is a square. Thus we
are reduced to the case when $G$ is a sum of $k$ copies of
$\BZ/2\BZ$. In this case all possible values of the  quadratic
form $q$ are $\pm 1, \pm i$ and  since $\tau^+ (G,q) = 2^{k/2}$,
there is at least one non-identity $a\in G$ with $q(a) =1$. So the
subgroup generated by $a$ is isotropic. The proposition is proved.
\end{proof}



\section{Nilpotent modular categories}
\label{appnilp}

In this section we prove our main results, stated in 1.1, and
derive a few corollaries.

Recall the definitions of $\K_{ad}$ and $\K^{co}$ from 2.5.

\begin{proposition}
\label{max symmetric}
Let $\C$ be a nilpotent modular  category.
Then for any maximal symmetric subcategory $\K$ of $\C$
one has $(\K')_{ad} \subseteq \K$. Equivalently,
there is a grading of $\K'$ such that
$\K$ is the trivial component:
\begin{equation}
\label{grading of K'}
\K' = \oplus_{g\in G}\, \K'_g,\qquad \K'_1 = \K.
\end{equation}
\end{proposition}
\begin{proof}
The two conditions are equivalent since by
\cite{GN} the adjoint subcategory is the trivial component of the 
universal grading.

Let $\K$ be a symmetric subcategory of $\C$, 
i.e., such that $\K \subseteq \K'$.
Assume that $(\K')_{ad}$  is not contained in  $\K$.
It suffices to show that $\K$ is not maximal.

Let $\E = (\K^{co}  \cap (\K')_{ad})\vee \K$.
Clearly, $\K \subseteq \E \subseteq \K'$.
We have
\begin{eqnarray*}
\E' &=& ( (\K^{co}  \cap (\K')_{ad})\vee \K )' \\
&=&  \K' \cap ((\K^{co})'  \vee  ((\K')_{ad})') \\
&=&  \K'\cap ((\K')_{ad}  \vee  \K^{co})  \\
&=&  (\K' \cap  \K^{co}) \vee  (\K')_{ad},
\end{eqnarray*}
where we used the modular law  of the lattice $L(\C)$ from
Lemma~\ref{Dedekind}. Since $\K \subseteq  \K' \cap  \K^{co}$ and
$\K^{co}  \cap (\K')_{ad} \subseteq  (\K')_{ad}$ we see that $\E
\subseteq \E'$, i.e., $\E$ is symmetric.


Let $n$ be the largest positive integer such that $(\K')^{(n)}
\not\subseteq \K$.  Such $n$ exists by our assumption and the
nilpotency of $\K'$. We claim that $(\K')^{(n)} \subseteq
\K^{co}$. Indeed,
\[
(\K')^{(n)} \subseteq  ((\K')^{(n+1)})^{co}  \subseteq \K^{co}
\]
since $\D \subseteq (\D_{ad})^{co}$ for every subcategory 
$\D \subseteq \C$.
Therefore, $\K^{co}  \cap (\K')^{(n)} = (\K')^{(n)}$ 
is not contained in $\K$ and
\[
\K \subsetneq (\K^{co}  \cap (\K')^{(n)})\vee \K 
\subseteq (\K^{co}  \cap (\K')_{ad})\vee \K  =\E,
\]
which completes the proof.
\end{proof}

Recall that in a fusion category whose dimension is  an {\em odd}
integer the dimensions of all objects are automatically integers
\cite[Corollary 3.11]{GN}.

\begin{corollary}
\label{nilp mod cat are gr thr}
A nilpotent modular category $\C$ with integral dimensions of simple
objects is group-theoretical.
\end{corollary}
\begin{proof}
This follows immediately from 
Corollary~\ref{useful grthr characterization}
and Proposition~\ref{max symmetric}.
\end{proof}

\begin{remark}
\label{pointed quotient+grading}
It follows from Corollary~\ref{group theor criterion}
that a nilpotent modular category $\C$ with integral dimensions
of simple objects contains an isotropic subcategory $\E$ such
that $(\E')_{ad}\subseteq \E$. 
 The corresponding grading
\begin{equation}
\label{grading I wanted}
\E' =\oplus_{h\in H}\, \E'_h,\qquad \E'_1=\E,
\end{equation}
gives rise to a non-degenerate quadratic form $q$ on $H$
defined by $q(h) = \theta_V$ for any non-zero $V\in \C_h$.
We have a braided equivalence $\bar{\E'}\cong \C(H, q)$.

We may assume that $\E$ is maximal among
isotropic subcategories of $\C$. In this case,
Proposition~\ref{canonical modularization} implies
that the isomorphism class of the above  metric group $(H, q)$
does not depend on the choice of
the maximal isotropic subcategory $\E$.
\end{remark}

\begin{corollary}\label{charge}
The central charge of a modular nilpotent category 
with integer dimensions
of objects is always an $8$th root of $1$.
Moreover, the central charge 
of a modular $p-$category is $\pm 1$ if $p=1\,mod\, 4$
and $\pm 1$, $\pm i$ if $p=3\,mod\,4$. 
The central charge of a modular $p-$category of
dimension $p^{2k}$, $k\in \BZ^+$ with odd $p$ is $\pm 1$.
\end{corollary}
\begin{proof}
By Remark~\ref{pointed quotient+grading} and 
Theorem~\ref{tau after mod}
the central charge always equals the
central charge of some pointed category,
so the first claim follows
from Examples~\ref{peven}-\ref{podd}.
The second and third claims follow from Example~\ref{podd}.
\end{proof}

\begin{theorem}
\label{factor in a double}\label{main2}
Let $\C$ be a  modular category with
integral dimensions of simple objects. Then $\C$  is nilpotent
if and only if there exists a pointed
modular category $\M$ such that $\C \bt \M$ is equivalent
(as a braided fusion category)  to $\Z(\Vec_G^\omega)$, where $G$
is a nilpotent group.
\end{theorem}
\begin{proof}
Note that for a nilpotent group $G$ the category $\Z(\Vec_G^\omega)$
is a tensor product of modular $p$-categories and, hence, is nilpotent.
So if  $\C \bt \M \cong \Z(\Vec_G^\omega)$ then $\C$ is nilpotent
(as a subcategory of a nilpotent category).

Let us prove the converse implication.
Pick an isotropic subcategory $\E \subset \C$
such that $(\E')_{ad}\subseteq \E$ (such a subcategory
exists by  Remark~\ref{pointed quotient+grading}).
There is a metric group $(H,q)$ such that $\bar{\E'}\cong \C(H,q)$.
Let  $\E' = \oplus_{h\in H}\, \E'_h$ , where $\E_1 = \E$ 
be the corresponding grading from \eqref{grading I wanted}.

Let $\M$ be the reversed category of $\bar{\E'}$
(i.e., with the opposite braiding and twist).
Then $\M \cong \C(H,q^{-1})$ and  
$\xi(\M) = \xi(\C(H,q))^{-1} = \xi(\C)^{-1}$
by Theorem~\ref{tau after mod}.

The  modular category $\C_{new} = \C \bt \M$ 
is nilpotent and $\xi(\C_{new}) =1$.
The category $\E_{new}:= \oplus_{h\in H}\, \E_h\boxtimes h$
is a Lagrangian subcategory of $\C_{new}$ and
the required statement  follows from
Theorem~\ref{main1}.
\end{proof}

Let $p$ be a prime number.

\begin{theorem}
\label{main3}
A modular category $\C$
is equivalent to the center of a fusion category of the form
$\Vec_G^\omega$ with $G$ being a $p$-group 
if and only if it has the following properties:
\begin{enumerate}
\item[(i)] the Frobenius-Perron dimension of $\C$ is 
$p^{2n}$ for some $n\in \mathbb{Z^+}$,
\item[(ii)] the dimension of every simple object of $\C$ is an integer,
\item[(iii)] the multiplicative central charge of $\C$ is $1$.
\end{enumerate}
\end{theorem}
\begin{proof}
It is clear that for any finite $p$-group $G$ and 
$\omega\in Z^3(G,\, k^\times)$
the modular category $\Z(\Vec_G^\omega)$ 
satisfies properties (i) and (ii).
The central charge of $\Z(\Vec_G^\omega)$ 
equals  $1$ by \cite[Theorem 1.2]{Mu4}.

Let us  prove the converse. Suppose that $\C$ satisfies conditions
(i), (ii), and (iii). 
Let $\E$ be an isotropic subcategory of $\C$ such
that $(\E')_{ad}\subseteq \E$ (such an $\E$ exists by
Remark~\ref{pointed quotient+grading}). There is a grading
$\E' =\oplus_{h\in H}\, \E'_h$ with $\E'_1=\E$
and $\theta$ being constant on each $\E'_h,\, h\in H$. 
Note that $H$ is a metric $p$-group whose
order is a square. By Proposition~\ref{basis}
it contains a Lagrangian subgroup $H_0$, whence
$\oplus_{h\in H_0}\, \E'_h$ is a Lagrangian subcategory of $\C$.

Thus, $\C \cong \Vec_G^\omega$  for some $G$ and $\omega$
by Theorem~\ref{main1}.
Since $|G|^2= \dim(\Vec_G^\omega) =\dim(\C)$ it follows
that $G$ is a $p$-group.
\end{proof}

Finally, we apply our results  to show that certain
fusion categories (more precisely, representation categories of certain
semisimple quasi-Hopf algebras) are group-theoretical
and to obtain a categorical analogue of the Sylow decomposition
of nilpotent groups.

\begin{corollary}
\label{cor0}
Let $\C$ be a fusion category with integral dimensions of simple
objects and such that $\Z(\C)$ is nilpotent.
Then $\C$ is group-theoretical.
\end{corollary}
\begin{proof}
By Corollary~\ref{nilp mod cat are gr thr}  the category
$\Z(\C)$ is group-theoretical.
Hence, $\C \bt \C^\rev$ is group theoretical
(as a dual category of $\Z(\C)$, see \cite{ENO}).
Therefore, $\C$ is group-theoretical (as a fusion subcategory
of $\C \bt \C^\rev$).
\end{proof}

\begin{corollary}
\label{cor1} Let $\C$ be a fusion category of dimension $p^n,\,
n\in \mathbb{Z^+}$, such that all objects of $\C$ have integer
dimension (this is automatic if $p>2$). Then $\C$ is
group-theoretical.
\end{corollary}

In other words, semisimple quasi-Hopf algebras of dimension $p^n$
are group-theoretical.

\begin{remark}
Semisimple Hopf algebras of dimension $p^n$ were studied by
several authors, see e.g., \cite{EG2}, \cite{Kash}, \cite{Ma1},
\cite{Ma2}, \cite{MW}, \cite{Z}.
\end{remark}

 From Corollary~\ref{nilp mod cat are gr thr} we obtain
the following Sylow decomposition.

\begin{theorem}\label{cor2}
Let $\C$ be a braided nilpotent fusion category  such that all
objects of $\C$ have integer dimension. Then $\C$ is
group-theoretical and has a decomposition into
a tensor product of braided fusion categories of prime power
dimension. If the factors are chosen
in such a way that their dimensions are relatively prime,
then such a decomposition is unique up to a permutation of factors.
\end{theorem}
\begin{proof}
It was shown in \cite[Theorem 6.11]{GN} that the center of a
braided nilpotent fusion category is nilpotent. Hence,
$\mathcal{Z}(\C)$ is group-theoretical by
Corollary~\ref{nilp mod cat are gr thr}.
Since $\C$ is equivalent to a subcategory of $\mathcal{Z}(\C)$, it
is group-theoretical by \cite[Proposition 8.44]{ENO}. This
means that there is a group $G$ and $\omega\in Z^3(G,\, k^*)$ such
that  $\C$ is dual to $\Vec_G^\omega$ with respect to some
indecomposable module category. The group $G$ is necessarily
nilpotent since $\Rep(G)\subseteq \mathcal{Z}(\Vec_G^\omega) \cong
\mathcal{Z}(\C)$. Hence, $G$ is isomorphic to a direct product of
its Sylow $p$-subgroups, $G = G_1 \times \cdots \times G_n$, and so
$\Vec_G^\omega$ is equivalent to a tensor product of
$p$-categories. It follows from \cite[Proposition 8.55]{ENO} that
the dual category $\C$ is also a product of fusion $p$-categories,
as desired.


Now suppose that $\C$ is decomposed into factors of 
prime power Frobenius-Perron dimension,
$\C \simeq \boxtimes_p \C_p$. 
It is easy to see that the objects from $\C_p \subset \C$ are
characterized by the following property:
\begin{equation}
\label{*}
X\in \C_p \mbox{ if and only if there exists }
k\in \BZ^+ \mbox{ such that }\Hom(\be, X^{\otimes^{p^k}})\ne 0.
\end{equation}
This shows that the decomposition in question is unique.
\end{proof}

\begin{remark}
\label{exponent leq 2}
Let $\C$ be a nilpotent modular category with integral dimensions 
of simple objects.
We already mentioned in the introduction that 
the choice of a tensor complement $\M$
satisfying $\C \bt \M \cong \Vec_G^\omega$ is not unique. In the proof
of Theorem~\ref{main2} such $\M$ can be chosen canonically 
as the category opposite to the canonical modularization 
corresponding to a maximal isotropic subcategory of $\C$, 
see  Proposition~\ref{canonical modularization}.

Another canonical way is to choose an $\M$ of minimal possible dimension.
This is done as follows.
By Theorem~\ref{cor2}, we have $\C = \boxtimes_p\, \C_p$ and
$\M = \boxtimes_p\, \M_p$, where $\C_p, \M_p$ are modular $p$-categories.
By Theorem~\ref{main3}, $\M_p$ has to be chosen in such a way
that $\dim(\C_p)\dim(\M_p)$ is a square and $\xi(\M_p)= \xi(\M_p)^{-1}$.
It follows from 
Examples~\ref{peven}, \ref{podd} and Corollary~\ref{charge}
that there is a unique such choice of $\M_p$ with minimal $\dim(\M_p)$,
in which case $\dim(\M_p) \in \{ 1,\, p,\, p^2\}$ for odd $p$ and
$\dim(\M_2)\in\{1, 2, 4, 8\}$.
\end{remark}

\begin{theorem}\label{cor3}
Let $\C$ be a braided nilpotent fusion category. 
Then $\C$ has a unique decomposition into
a tensor product of braided fusion categories of prime power dimension.
\end{theorem}

\begin{proof} According to Theorem \ref{cor2} 
the result is true if the dimensions of simple objects
of $\C$ are integers. 
In general, define subcategories $\C_p\subset \C$ by condition
\eqref{*} above.
For a simple object $X\in \C$ it is known (see \cite{GN}) 
that $\FPdim(X)=\sqrt{N}$, $N\in \BN$.
Thus $X\boxtimes X\in \C \boxtimes \C$ has an integer dimension. 
The category $\C \boxtimes \C$
contains a fusion subcategory $(\C \boxtimes \C)^{int}$ 
consisting of all objects with integer dimension,
see \cite{GN}. We can apply Theorem \ref{cor2} 
to the category $(\C \boxtimes \C)^{int}$ and obtain
a unique decomposition $X=\otimes_p X_p$ with $X_p\in \C_p$. 
The theorem is proved.
\end{proof}

\begin{corollary} Let $\C$ be a braided nilpotent fusion category. 
Assume that $X\in \C$ is simple
and its dimension is not integer. Then $\FPdim(X)\in \sqrt{2}\BZ$.
\end{corollary}
\begin{proof} This follows immediately from Theorem \ref{cor3} 
since if a category of prime power
Frobenius-Perron dimension $p^k$ contains an object 
of a non-integer dimension then $p=2$,
see \cite{ENO}.
\end{proof}

\begin{example} It is easy to see that 
the Tambara-Yamagami categories from \cite{TY} are
nilpotent and indecomposable into a tensor product. 
Thus Theorem \ref{cor3} implies
that if such a category admits a braiding, 
then its dimension should be a power of 2 (since the
dimension of a Tambara-Yamagami category is always divisible by 2). 
A stronger result is contained in \cite{Sie}.
\end{example}


\bibliographystyle{ams-alpha}

\end{document}